\documentclass[final,3p,times]{elsarticle}

\usepackage[utf8]{inputenc}
\usepackage{siunitx}
\usepackage{amsmath}
\usepackage{array}
\usepackage{amsthm}
\usepackage{amsfonts}
\usepackage{graphicx}
\usepackage{float}
\usepackage{bm}
\usepackage[final]{microtype}
\usepackage{mathtools}
\mathtoolsset{showonlyrefs}

\usepackage{todonotes}

\usepackage{hyperref}
\hypersetup{
	hidelinks,
    colorlinks=black,
    linkcolor=black,
    filecolor=black,      
    urlcolor=black,
}

\newtheorem{theorem}{Theorem}

\newtheorem{definition}{Definition}
\newtheorem{example}{Example}
\newtheorem{lemma}{Lemma}

\newtheorem{remark}{Remark}
\newtheorem{problem}{Problem}

\newcommand{\sign}[1]{\mbox{sign}(#1)}
\newcommand{\sgn}[1]{\lfloor#1\rceil}
\newcommand{\real}[1]{\mbox{Re}(#1)}
\renewcommand{\exp}[1]{\mbox{exp}(#1)}

\def\Red#1{\textcolor{red}{#1}}

\journal{ArXiv}

\begin{document}

\begin{frontmatter}



\title{A predefined-time first-order exact differentiator based on time-varying gains
\footnote{\Red{This is the preprint version of the accepted manuscript: Aldana-López R., Gómez-Gutiérrez D., Trujillo M.A., Navarro-Gutiérrez M.,
Ruiz-León J., Becerra H. M. A predefined-time first-order exact differentiator based on time-varying gains. Int J
Robust Nonlinear Control. 2021; 1–13. DOI: 10.1002/rnc.5536. 
\textbf{Please cite the publisher's version}. For the publisher's version and full citation details see:
\url{https://doi.org/10.1002/rnc.5536}. This article may be used for non-commercial purposes in accordance with Wiley Terms and Conditions for Use of Self-Archived Versions. 
}}
}

\author[label0]{R. Aldana-López}
\author[label1,label2]{D. Gómez-Gutiérrez}
\ead{David.Gomez.G@ieee.org}
\author[label3]{M. A. Trujillo}
\author[label2]{M. Navarro-Gutiérrez} 
\author[label3]{J. Ruiz-León}
\author[label4]{H. M. Becerra}

\address[label0]{Departamento de Informatica e Ingenieria de Sistemas (DIIS), University of Zaragoza, María de Luna, s/n, 50018, Zaragoza, Spain.}
\address[label1]{Multi-agent autonomous systems lab, Intel Labs, Intel Tecnología de M\'exico, Av. del Bosque 1001, Colonia El Bajío, Zapopan, 45019, Jalisco, M\'exico.}
\address[label2]{Tecnologico de Monterrey, Escuela de Ingenier\'ia y Ciencias, Av. General Ram\'on Corona 2514, Zapopan, 45201, Jalisco, M\'exico.}
\address[label3]{CINVESTAV, Unidad Guadalajara, Av. del Bosque 1145, Colonia el Bajío, Zapopan , 45019, Jalisco, Mexico.}
\address[label4]{Computer Science Dept., Centro de Investigaci\'{o}n en Matem\'{a}ticas (CIMAT), Jalisco S/N, Col. Valenciana, 36023, Guanajuato, Mexico}

\begin{abstract}
Recently, a first-order differentiator based on time-varying gains was introduced in the literature, in its non recursive form,  for a class of differentiable signals $y(t)$, satisfying $|\ddot{y}(t)|\leq L(t-t_0)$, for a known function $L(t-t_0)$, such that $\frac{1}{L(t-t_0)}\left|\frac{d {L}(t-t_0)}{dt}\right|\leq M$ with a known constant $M$. It has been shown that such differentiator is globally finite-time convergent. In this paper, we redesign such an algorithm, using time base generators (a class of time-varying gains), to obtain a differentiator algorithm for the same class of signals, with guaranteed convergence before a desired time, i.e., with fixed-time convergence with an a priori user-defined upper bound for the settling time. Thus, our approach can be applied for scenarios under time-constraints.

We present numerical examples exposing the contribution with respect to state-of-the-art algorithms.
\end{abstract}


\end{frontmatter}

\section{Introduction}

The exact differentiator problem is a relevant problem in control theory, that has recently received a great deal of attention~\cite{Levant2003,Moreno2012,Pico2013,Levant2019,Levant2018GloballyGains,Sanchez2016construction,Cruz-Zavala2011,Carvajal-Rubio2019OnDifferentiators}, as it allows to obtain in a finite-time the derivative of a measurable signal and can be applied, among other problems, to the unknown input observer problem~\cite{Fridman2008Higher-orderSystems,Angulo2014OnDifferentiators,Nehaoua2014AnDynamics}; fault detection and isolation~\cite{Rios2015FaultApproach,Rios2012FaultMultiple-observer,Kommuri2016AVehicles}; active disturbance rejection~\cite{FerreiraDeLoza2015OutputIdentification}; and it is an essential part in the universal controller for single-input-single-output systems~\cite{Levant2001HigherDifferentiation,Angulo2009OnHOSMs,Levant2001UniversalConvergence,ShihongDing2015NewControllers}.

For the case where the $n$-th derivative of the input signal is Lipschitz with a known Lipschitz constant, an arbitrary order exact differentiator algorithm has been proposed~\cite{Levant2003}. Lyapunov functions for such algorithm were proposed for the first-order differentiator~\cite{Moreno2012}, for the second-order differentiator~\cite{Sanchez2016construction} and for the arbitrary order case~\cite{Cruz2018}.

For the case where the input signal $y(t)$ is an $n$ times differentiable signal, satisfying $|\frac{d^{n+1}y(t)}{dt^{n+1}}|\leq L(t-t_0)$, for a known function $L(t-t_0)$, such that $\frac{1}{L(t-t_0)}\left|\frac{d {L}(t-t_0)}{dt}\right|\leq M$ with a known constant $M$, Levant and Livne~\cite{Levant2018GloballyGains} introduced an arbitrary order differentiator using $L(t-t_0)$ as a time-varying gain. A Lyapunov function for such algorithm has been proposed~\cite{Moreno2017LevantsGain}. It can be concluded that such algorithms are finite-time convergent according to existing results~\cite{Moreno2017LevantsGain,Cruz2018}. 

To apply differentiator algorithms for scenarios with time constraints, i.e., with guaranteed convergence before a user defined-time, there has been some effort to design differentiator algorithms with uniform convergence independent of the initial condition, i.e., with fixed-time convergence~\cite{Aldana-Lopez2018,Sanchez-Torres2018,Jimenez2019,aldana2019design,Sanchez-Torres2020ASystems,DiegoSanchez-Torres2012AnOperators}, where there exists an upper bound for the settling-time (\textit{UBST}) function that is independent of the initial condition~\cite{Angulo2013RobustDifferentiator,Fraguela2012DesignNoise,Cruz-Zavala2011,Efimov2017SwitchedConvergence}. Of greater interest is when such \textit{UBST} is known since the desired convergence time can be set a priori (predefined) by the user. For the case, where $L(t-t_0)$ is constant, first-order algorithms with predefined convergence have been proposed~\cite{Cruz-Zavala2011,Fraguela2012DesignNoise,Seeber2020ExactBound}. However, the resulting predefined \textit{UBST} is conservative (see e.g. \cite[Section~5]{Cruz-Zavala2011} where the estimate of the \textit{UBST} is approx. 217s, but the simulated one is approx. 2s). However, to our best knowledge, no predefined-time algorithm exists for the case where
$L(t-t_0)$ is time-varying satisfying $\frac{1}{L(t-t_0)}\left|\frac{d {L}(t-t_0)}{dt}\right|\leq M$ with a known constant $M$.

There are different scenarios where real-time constraints need to be satisfied and fixed-time convergence is an important property for those, for instance: In missile guidance~\cite{Zarchan2012}, stabilization in a desired time is required by the impact time control guidance laws~\cite{Song2017}. In fault detection, isolation, and recovery schemes~\cite{Tabatabaeipour2014CalculationAnalysis}, an unrecoverable mode may be reached if failing to recover from the fault on time. In hybrid dynamical systems, it is a common need that the observer (resp. controller) stabilizes the observation error (resp. tracking error) before the next switching occurs~\cite{defoort2011,Gomez2015}. In the frequency control of an interconnected power network, besides the frequency deviation, it is also of interest to know how long the frequency stays out of the bounds~\cite{Mishra2018}. Similarly, for chaos suppression in power systems, the convergence time is an essential performance specification~\cite{Ni2016}, since oscillations are acceptable if they can be damped within a limited time.

In this work, we propose redesigning the first-order differentiator, in its non-recursive form, proposed by Levant and Livne~\cite{Levant2018GloballyGains}. To this aim, we use a class of Time Base Generators (\textit{TBG})~\cite{Morasso1997ANetworks,aldana2019design}. However, contrary to the referred work~\cite{Levant2018GloballyGains}, in our approach we obtain guaranteed convergence at a desired time predefined by the user, i.e., fixed-time convergence, with a predefined \textit{UBST}.

The contribution with respect to other autonomous predefined-time first-order differentiators, such as~\cite{Cruz-Zavala2011,Seeber2020ExactBound} is two-fold. First, the class of signals that we can differentiate is wider. Second, whereas in such autonomous algorithms the desired \textit{UBST} is very conservative, in our approach the slack of the \textit{UBST} is significantly reduced. This results in a convergence where the maximum value of the differentiation error signals is significantly lower, as it is illustrated by numerical examples.

To our best knowledge, the closest work to our approach is~\cite{Holloway2019}, as the same class of time-varying gains is used. However, the results in the referred work~\cite{Holloway2019} can only be applied to the first order differentiator problem for signals with zero second derivative. Moreover, in that work~\cite{Holloway2019}, in every nonzero trajectory, the time-varying gain tends to infinity as the zero error is reached. In our approach, for all finite initial conditions, the error is reached before the singularity in the \textit{TBG} gain occurs.

The rest of the manuscript is organized as follows. In Section~\ref{Sec:Prelim}, we recall the first-order differentiator, in its non-recursive form, that will be redesigned~\cite{Levant2018GloballyGains}, and present basic concepts on fixed-time stability and time-scale transformations. In Section~\ref{Sec:Main}, we present the problem formulation and the proposed predefined-time exact differentiator algorithm. In Section~\ref{Sec:Sim}, we show numerical examples to illustrate our approach; exposing the main advantages with respect to the state-of-the art. Finally, in Section~\ref{Sec:Conclu} we present some concluding remarks and suggest some proposed future work.
\vspace{0.3cm}

\noindent \textbf{Notation:}
$\mathbb{R}$ is the set of real numbers, $\mathbb{R}_+=\{x\in\mathbb{R}\,:\,x\geq0\}$. For $x\in\mathbb{R}$, $\sgn{x}^\alpha = |x|^\alpha \mbox{sign}(x)$, if $\alpha\neq0$ and $\sgn{x}^\alpha = \mbox{sign}(x)$ if $\alpha=0$. 
For a function $\phi:\mathcal{I}\to\mathcal{J}$, its reciprocal $\phi(\tau)^{-1}$, $\tau\in\mathcal{I}$,  is such that $\phi(\tau)^{-1}\phi(\tau)=1$ and its inverse function $\phi^{-1}(t)$, $t\in\mathcal{J}$, is such that $\phi(\phi^{-1}(t))=t$. For functions $\phi,\psi:\mathbb{R}\to\mathbb{R}$, $\phi\circ\psi(t)$ denotes the composition $\phi(\psi(t))$. Given a matrix $A\in\mathbb{R}^n\times\mathbb{R}^m$, $A^T$ represents the matrix transpose of $A$.  

\section{Problem statement and preliminaries}
\label{Sec:Prelim}

\subsection{Problem statement: Predefined-time first-order exact differentiator}

\begin{problem}[The predefined-time exact first order differentiator problem]
\label{Problem}
Considering a user-defined time $T_c$ and a differentiable signal $y(t)\in \mathbb{R}$ such that $\dot{y}(t)$ is Lipschitz and $\left|\ddot{y}(t)\right|\leq L(t-t_0)$, for all $t\geq t_0$, with $L(t-t_0)$ satisfying $\frac{1}{L(t-t_0)}\left|\frac{d{L}(t-t_0)}{dt}\right|\leq M$ for a known constant $M$, the problem consists in accurately obtaining the functions $y(t)$ and $\dot{y}(t)$, for all time $t\geq t_0+T_c$. The set of admissible signals $y(t)$ is denoted as $\mathcal{Y}$.
\end{problem}
Solving this problem enables the application to control problems with time constraints. To solve it, we propose, for an a priori given $T_c>0$ (a desired convergence time), to design functions $h_i(w,t;T_c)$, $i=1,2$ such that, with the algorithm:
\begin{align}
w=&z_0-y(t)\\
\dot{z}_{0}=&-h_1(w,t;T_c)+z_1, \\
\dot{z}_1=&-h_2(w,t;T_c),   \label{Eq:Diff} 
\end{align}
we obtain that, for all $t\geq t_0+T_c$ and every initial condition $z_0(t_0)$ and $z_1(t_0)$, $z_i=\frac{d^iy(t)}{dt^i}$ for $i=0,1$.  

The solutions of~\eqref{Eq:Diff} are understood in the sense of Filippov~\cite{Cortes2008DiscontinuousSystems}.

\begin{remark}
Notice that with $h_i(w,t;T_c)$, $i=1,2$ independent $T_c$, the algorithm~\eqref{Eq:Diff} has the structure of the differentiator in~\cite{Levant2018GloballyGains}, in its non-recursive form. Here our aim is to guarantee exact convergence, before a user-defined time given by $T_c$, and regardless of the initial condition. A feature not present in the base approach~\cite{Levant2018GloballyGains}.
\end{remark}

\subsection{Fixed-time stability}

To analyze the convergence of the differentiators, we analyze the stability of the differentiation error dynamics given by
\begin{align}
\dot{e}_{1}=&-h_1(e_1,t;T_c)+e_2, \\
\dot{e}_2=&-h_2(e_1,t;T_c)-\ddot{y}(t),   \label{Eq:DiffErr} 
\end{align}
where $h_i:\mathbb{R}^n\times\mathbb{R}_+\to\mathbb{R}^n$, $i=1,2$, is some function that is continuous on $x$ (except, perhaps, at the origin), and continuous almost everywhere on $t$ and $|\ddot{y}(t)|\leq L(t-t_0)$.

We assume that $h_i(w,t;T_c)$, $i=1,2$ are such that the origin of~\eqref{Eq:DiffErr} is asymptotically stable and, perhaps except at sets of measure zero, \eqref{Eq:DiffErr} has the properties of existence and uniqueness of solutions in forward-time on the interval $[t_0,+\infty)$~\cite[Proposition~5]{Cortes2008DiscontinuousSystems}. 

The set of admissible $y(t)$ functions, on the interval $[t_0,\mathbf{t}]$ with $t_0<\mathbf{t}$ is denoted by $\mathcal{Y}_{[t_0,\mathbf{t}]}$.
The solution of \eqref{Eq:DiffErr} for $t\in [t_0,\mathbf{t}]$, with signal $y_{[t_0,\mathbf{t}]}$ (i.e. the restriction of the map $y(t)$ to $[t_0,\mathbf{t}]$) and initial condition $e_0$ is denoted by $e(t;e_0,t_0,y_{[t_0,\mathbf{t}]})$, and the initial state is given by $e(t_0;x_0,t_0,\cdot) = e_0$.

We assume that the origin is the unique equilibrium point of~\eqref{Eq:DiffErr}.  Note that because $h_i(w,t;T_c)$, $i=1,2$ may be discontinuous at a set of measure zero, system \eqref{Eq:DiffErr} can have an equilibrium point at the origin despite the presence of disturbances.

For the system in Eq.~\eqref{Eq:DiffErr}, its \textit{settling-time function} $T(e_0, t_0)$ for the initial state $e_0 \in \mathbb R^n$ and the initial time $t_0 \geq 0$ is defined as:
\begin{equation}
T(e_0,t_0)=\inf\left\{\xi\geq t_0:\forall y_{[t_0,\infty)}\in\mathcal{Y}_{[t_0,\infty)}, \lim_{t\to\xi}e(t;e_0,t_0,y_{[t_0,\infty)})=0\right\}-t_0.
\end{equation}

For simplicity, in the rest of this paper we write ``stable" instead of ``the origin is globally stable". With this shorthand, we can introduce the notion of fixed-time stability.

\begin{definition}(Finite-time stability) 
\label{def:finite}
System \eqref{Eq:DiffErr} is  \textit{finite-time stable} if it is asymptotically stable~\cite{Khalil2002NonlinearSystems} and for every initial state $e_0 \in \mathbb R^n$ and the initial time $t_0 \geq 0$, the settling-time function $T(e_0,t_0)$ is finite.
\end{definition}

\begin{definition} (Fixed-time stability)
\label{def:fixed}
System \eqref{Eq:DiffErr} is  \textit{fixed-time stable} if it is asymptotically stable~\cite{Khalil2002NonlinearSystems} and the settling-time function $T(e_0,t_0)$ is bounded on  $\mathbb{R}^n\times\mathbb{R}_+$, i.e. there exists $T_{{\max}}<+\infty$ such that $T(e_0,t_0)\leq T_{\max}$, for all $t_0\in\mathbb{R}_+$ and $e_0\in\mathbb{R}^n$. The quantity $T_{\max}$ is called an Upper Bound of the Settling Time (\textit{UBST}) of the system~\eqref{Eq:DiffErr}.
\end{definition}

\subsection{Levant's first-order differentiator with time-varying gains}

\begin{theorem}[\cite{Levant2018GloballyGains}]
\label{Th:Levant}
Given a differentiable signal $y(t)\in \mathbb{R}$ such that $\dot{y}(t)$ is Lipschitz and $\left|\ddot{y}(t)\right|\leq L(t-t_0)$, for all $t\geq t_0$, with $L(t-t_0)$ satisfying $\frac{1}{L(t-t_0)}\left|\frac{d {L}(t-t_0)}{dt}\right|\leq M$ for a known constant $M$. 
The algorithm:
\begin{align}
w=&z_0-y(t)\\
\dot{z}_{0}=&-\phi_1(w;M,L(t-t_0))+z_1, \\
\dot{z}_1=&-\phi_2(w;M,L(t-t_0)),   \label{Eq:LevDiff} 
\end{align}
is a first-order exact differentiator, i.e., there exists a finite-time $\hat{t}$, such that $z_0=y(t)$ and $z_1=\dot{y}(t)$ for all time $t\geq t_0+\hat{t}$, where
\begin{equation}
\label{Eq:Levant1}
\phi_1(w;M,L(t-t_0)):=\lambda_1 L(t-t_0)^{\frac{1}{2}}\sgn{w}^{\frac{1}{2}}+\mu_1M w 
\end{equation}
and 
\begin{equation}
\label{Eq:Levant2}
\phi_2(w;M,L(t-t_0)):=\lambda_0L(t-t_0)\sign{w}+\lambda_1\mu_0L(t-t_0)^{\frac{1}{2}}M\sgn{w}^{\frac{1}{2}}+\mu_0\mu_1M^2w,
\end{equation}
where $\mu_i$, $i=0,1$ are such that $s^2+\mu_1s+\mu_0\mu_1$ is a Hurwitz polynomial and  $\lambda_i$, $i=0,1$ are suitable positive constants.

In other words, the system 
\begin{align}
\dot{e}_{1}=&-\phi_1(e_1;M,L(t-t_0))+e_2, \\
\dot{e}_2=&-\phi_2(e_1;M,L(t-t_0))-\ddot{y}(t),   \label{Eq:LevDiffErr} 
\end{align}
where $e_1=z_0-y(t)$ and $e_2=z_1-\dot{y}(t)$ (i.e., system~\eqref{Eq:DiffErr} with $h_1(w,t;T_c)=\phi_1(w;M,L(t-t_0))$ and $h_2(w,t;T_c)=\phi_2(w;M,L(t-t_0))$), is finite-time stable.
\end{theorem}

For the sake of simplicity, throughout the manuscript we consider $\lambda_0=1.1$, $\lambda_1=1.5$, $\mu_0=2$ and $\mu_1=3$ as suggested in~\cite{Levant2018GloballyGains}.

\begin{remark}
The case where $y(t)$ satisfies $|\ddot{y}(t)|\leq L(t-t_0)$ with constant $L(t-t_0)$ was proposed in~\cite{Castillo2018Super-TwistingPerturbations}, where it was shown, using Lyapunov analysis, that~\eqref{Eq:LevDiffErr} is finite-time stable with a settling time that is an unbounded function of the initial conditions.
\end{remark}

\subsection{Time-scale transformations}

The trajectories corresponding to the system solutions of~\eqref{Eq:DiffErr} are interpreted, in the sense of differential geometry~\cite{Kuhnel2015DifferentialGeometry}, as regular parametrized curves~\cite{Pico2013,aldana2019design}. Since we apply regular parameter transformations over the time variable, then without ambiguity, this reparametrization is sometimes referred to as time-scale transformation.

\begin{definition}(Regular parametrized curve~\cite[Definition~2.1]{Kuhnel2015DifferentialGeometry})
\label{Def:RegularParamCurve}
A regular parametrized curve, with parameter $t$, is a $C^1(\mathcal{I})$ immersion $c: \mathcal{I}\to \mathbb{R}$, defined on a real interval $\mathcal{I} \subseteq \mathbb{R}$. This means that $\frac{dc}{dt}\neq 0$ holds everywhere.
\end{definition}

\begin{definition}(Regular curve~\cite[Pg.~8]{Kuhnel2015DifferentialGeometry})
\label{Def:RegularCurve}
A regular curve is an equivalence class of regular parametrized curves, where the equivalence relation is given by regular (orientation preserving) parameter transformations $\varphi$, where $\varphi:~\mathcal{I}~\to~\mathcal{I}'$ is $C^1(\mathcal{I})$, bijective and $\frac{d\varphi}{dt}>0$. Therefore, if $c:\mathcal{I}\to\mathbb{R}$ is a regular parametrized curve and $\varphi:\mathcal{I}\to \mathcal{I}'$ is a regular parameter transformation, then $c$  and  $c\circ\varphi:\mathcal{I}'\to\mathbb{R}$ are considered to be equivalent.
\end{definition}

\begin{remark}
We will apply time-scale transformations to asymptotically stable systems. Thus, the trajectories of the system are represented by regular parametrized curves with time interval $I$ spanning from the initial condition to the origin. Thus, for $x(t)$ and its reparametrization (time-scaling) $\tilde{x}(\tau)$, to belong to the same equivalence class of regular parametrized curves, it is necessary that $\lim_{t\to\inf I }x(t)=\lim_{\tau\to\inf I'}\tilde{x}(\tau)$ and $\lim_{t\to\sup I }x(t)=\lim_{\tau\to\sup I'}\tilde{x}(\tau)$, where the trajectory $\tilde{x}(\tau)$ is defined on the interval $I'\subseteq \mathbb{R}_+$. 
\end{remark}

\begin{lemma}[\cite{aldana2019design}]
\label{Lemma:ParTrans}
The bijective function $\varphi:[t_0,t_0+T_c)\to[0,+\infty)$ defined by $\tau=\varphi(t):=-\alpha^{-1}\ln(1-(t-t_0)/T_c)$, defines a parameter transformation with $t=\varphi^{-1}(\tau)=T_c(1-\exp{-\alpha\tau})+t_0$ as its inverse mapping.
\end{lemma}

\section{Main result}\label{Sec:ProbStatement}
\label{Sec:Main}

Next, we introduce our main result.

\begin{theorem}
\label{Theorem:FOD}
Let 
\begin{itemize}
    \item $\varphi(\tau)$ be chosen as in Lemma \ref{Lemma:ParTrans} with $\alpha>0$,
    \item $\kappa(t-t_0):=\left({\alpha(T_c- (t-t_0))}\right)^{-1}$,
    \item $\mathcal{L}(\tau)$ such that $\mathcal{L}(\varphi(t))=L(t-t_0)\kappa(t-t_0)^{-2}$, and
    \item $\mathcal{M}$ such that $\frac{1}{\mathcal{L}(\tau)}\left|\frac{d \mathcal{L}(\tau)}{d\tau}\right|\leq\mathcal{M}$.
\end{itemize}

Then, using the algorithm~\eqref{Eq:Diff} and selecting $h_i(w,t;T_c)$, $i=1,2$, as:  
\begin{equation}
\label{Eq:H1Th1}
    h_1(w,t;T_c)=\left\lbrace
    \begin{array}{cc}
        \kappa(t)\left( -\alpha w+\phi_1(w;\mathcal{M},\mathcal{L}(\varphi(t))) \right) & \text{for } t\in[0,T_c) \\
        \phi_1(w;M,L(t)) & \text{otherwise,}
    \end{array}
    \right.
\end{equation}

\begin{equation}
\label{Eq:H2Th1}
    h_2(w,t;T_c)=\left\lbrace
    \begin{array}{cc}
        \kappa(t)^2\left( \alpha^2 w +\phi_2(w;\mathcal{M},\mathcal{L}(\varphi(t))-\alpha\phi_1(w;\mathcal{M},\mathcal{L}(\varphi(t)))) \right) & \text{for } t\in[0,T_c) \\
        \phi_2(w;M,L(t)) & \text{otherwise,}
    \end{array}
    \right.
\end{equation}
solves Problem~\ref{Problem} with $T_c$ as the upper bound for the convergence time. Moreover, if $L(t-t_0)$ is such that the settling-time function of~\eqref{Eq:LevDiffErr} is an unbounded function of the initial condition, then, $T_c$ is the least \textit{UBST} of~\eqref{Eq:DiffErr}.
\end{theorem}

\begin{proof}
Let $e_1(t)=z_0-y(t)$ and $e_2=z_1-\dot{y}(t)$. The proof is divided in two parts. First we will show that $e_1(t)=0$ and $e_2(t)=0$ for $t\in [\hat{t},t_0+T_c)$ for some time $\hat{t}$. Afterwards, we show that the condition $e_1(t)=0$ and $e_2(t)=0$ is maintained for all $t>t_0+T_c$.
 
Consider the coordinate change $\epsilon_1=e_1$ and $\epsilon_2=\alpha e_1+\kappa(t-t_0)^{-1}e_2$, where the error dynamics is given in~\eqref{Eq:DiffErr}. Then, the dynamics under the coordinate change for $t\in[t_0,t_0+T_c)$ is
\begin{align}
\dot{\epsilon}_{1}=&\kappa(t-t_0)\left( -\phi_1(\epsilon_1;\mathcal{M},\mathcal{L}(\varphi(t)))+\epsilon_2\right), \\
\dot{\epsilon}_2=&\kappa(t-t_0)(-\phi_2(\epsilon_1;\mathcal{M},\mathcal{L}(\varphi(t)))-\kappa(t-t_0)^{-2}\ddot{y}(t)).   \label{Eq:Epsilon} 
\end{align}
Now, consider the parameter transformation given in Lemma~\ref{Lemma:ParTrans} and notice that $\mathcal{L}(\tau):=L(\psi(\tau))\rho(\tau)^{-2}$, where $\rho(\tau)=\left.\kappa(t-t_0)\right|_{t=\varphi(t)}=(\alpha T_c)^{-1}\exp{\alpha \tau}$ and $L(\psi(\tau))=\left. L(t-t_0)\right|_{t=\varphi(t)}$; and let $\epsilon=[\epsilon_1,\epsilon_2]^T$, then 
$$ \frac{d\epsilon}{d\tau}=\left.  \frac{d\epsilon}{dt}  \frac{dt}{d\tau}\right|_{t=\varphi^{-1}(\tau)}.$$
Since $\frac{dt}{d\tau}=\left. \kappa(t-t_0)^{-1}\right|_{t=\varphi^{-1}(\tau)}$, for $t\in[t_0+T_c)$, then the dynamics of~\eqref{Eq:Epsilon} in the new time $\tau$ is given by
\begin{align}
\frac{d\epsilon_1}{d\tau}=&-\phi_1(\epsilon_1;\mathcal{M},\mathcal{L}(\tau))+\epsilon_2, \\
\frac{d\epsilon_2}{d\tau}=&-\phi_2(\epsilon_1;\mathcal{M},\mathcal{L}(\tau))+\delta(\tau),   \label{Eq:EpsilonTau} 
\end{align}
where $\delta(\tau)=-\rho(\tau)^{-2}\left.[\ddot{y}(t)]\right|_{t=\varphi(t)}$.

Thus, the disturbance $\delta(\tau)$ satisfies $|\delta(\tau)|\leq \mathcal{L}(\tau)$. 
Notice that $\mathcal{L}(\tau)$ is such that 
$$\frac{1}{\mathcal{L}(\tau)}
\left|\frac{d \mathcal{L}(\tau)}{d\tau}\right|\leq\mathcal{M}.$$ 

Thus, according to Theorem~\ref{Th:Levant}, system~\eqref{Eq:EpsilonTau} is finite-time stable and has a settling time function $\mathcal{T}(\epsilon_0,\tau_0)$. Using Lemma~\ref{Lemma:ParTrans}, we can conclude that the settling-time function of~\eqref{Eq:DiffErr} is \begin{equation}
\label{Eq:SeetlingTimeDif}
    T(e_0,t_0)=\lim_{\tau\to\mathcal{T}(\epsilon_0,\tau_0)}(\varphi^{-1}(\tau)-t_0)=T_c(1-\exp{-\alpha\mathcal{T}(\epsilon_0,\tau_0)}).
\end{equation}
Thus, 
\begin{equation}
\label{Eq:STsup}
    \sup_{(e_0,t_0)\in\mathbb{R}^2\times\mathbb{R}}T(e_0,t_0)\leq T_c.
\end{equation}
Then, $e_1(t)=0$ and $e_2(t)=0$ for $t\in [\hat{t},t_0+T_c)$, where $\hat{t}=t_0+T_c(1-\exp{-\alpha\mathcal{T}(\epsilon_0,\tau_0)})$. 
Moreover, it follows from~\eqref{Eq:SeetlingTimeDif}, that the equality in~\eqref{Eq:STsup} holds when $\sup_{(e_0,t_0)\in\mathbb{R}^2\times\mathbb{R}_+}\mathcal{T}(e_0,t)=+\infty$.

The second part of the proof follows trivially from Theorem~\ref{Th:Levant}, because, for all $t\geq t_0+T_c$, the differentiation error dynamics is given by system~\eqref{Eq:LevDiffErr}.
\end{proof}

\begin{remark}
To our best knowledge, the closest work to our approach to provide predefined-time algorithms is~\cite{Holloway2019}. However, notice that the referred method~\cite{Holloway2019} can only be applied to the first-order exact differentiator problem for signals with zero second derivative. 
\end{remark}

\begin{remark}
\label{Rem:BoundedGains}
Similarly as in~\cite{Holloway2019}, the time-varying gain $\kappa(t-t_0)$ tends to infinity as the time approaches the predefined-time. Notice that in \cite[Section~II.A]{Holloway2019}, workarounds are suggested to maintain the time-varying gain bounded in practice. However, with such workarounds, for every initial condition only convergence to a neighbourhood of the origin is obtained. In fact, the size of such neighbourhood tends to infinity as $\|e_0\|\to+\infty$. 

Nonetheless, notice that in our approach, for any finite initial condition, $e_1(t)=0$ and $e_2(t)=0$ are obtained before the singularity in $\kappa(t-t_0)$ occurs. Thus, in practice different methods can be used to obtain predefined-time convergence with bounded gains. A naive approach is to choose a constant $T^*<T_c$ and to switch~\eqref{Eq:H1Th1} and~\eqref{Eq:H2Th1} at $T^*$ instead that at $T_c$. 
Notice that the switching occurs while $\kappa(t-t_0)$ is bounded. Also note that with such workaround, the differentiation error of our algorithm is still finite-time convergent. Moreover, there exists a neighbourhood of initial conditions around the origin, whose settling time is bounded by $T_c$, and the size of such neighbourhood can be set arbitrarily large with a suitable selection of the $\alpha$ and $\mathcal{M}$ parameters. This remark will be illustrated in Example~\ref{Ex:Krstic}.
\end{remark}

\begin{remark}
Compared with existing autonomous algorithms~\cite{Cruz-Zavala2011,Seeber2020ExactBound}, whose predefined \textit{UBST} is conservative (i.e., the slack between the least \textit{UBST} and the predefined one is large\footnote{See e.g., the example in \cite[Section~5]{Cruz-Zavala2011} where the estimate of the \textit{UBST} is approx. $217s$, but the simulated one is approx. $2s$; or the example in~\cite{Seeber2020ExactBound} where the \textit{UBST} is set as $T_c=1$, but in the simulations the convergence is bounded by $0.25s$. Note that no methodology is provided in such works to arbitrarily reduce such slack.}), we show that in our approach such slack can be significantly reduced. A consequence of reducing such slack is that the maximum differentiation error is significantly reduced, as it will be illustrated in Example~\ref{Ex:Sine}. Another advantage with respect to the algorithms proposed in~\cite{Cruz-Zavala2011,Seeber2020ExactBound} is that such algorithms can only be applied to the differentiator problem if the second derivative is bounded by a constant, a restriction that is not present in our approach.
\end{remark}

\section{Simulations and state-of-the-art comparison}
\label{Sec:Sim}

In this section we present numerical simulations to illustrate our methodology. Our first order differentiator algorithm presents guaranteed convergence before the desired time given by $T_c$. For the sake of simplicity, throughout the examples we consider $t_0=0$. The simulations below were created in OpenModelica using the Euler integration method with a step of $2e-4s$.

To illustrate the advantages with respect to the closest algorithm~\cite{Holloway2019}, consider the following example. Recall that the result in the referred work~\cite{Holloway2019} can only be applied to the first-order exact differentiator problem for signals with zero second derivative. Thus, we consider the problem of differentiating a linear function of time.

\begin{example}
\label{Ex:Krstic}
Consider the signal $y(t)=t+1$ which satisfies $\ddot{y}(t)=0$, and set $T_c=1$. For comparison, consider Example~1 in~\cite{Holloway2019}, i.e., algorithm~\eqref{Eq:Diff} with $h_i(w,t;t_c)=g_i(t-t_0,T)x_1$ where
\begin{align}
    g_1(t-t_0,T)=&l_1+2(m+2)\kappa(t-t_0;T_c),\\
    g_2(t-t_0,T)=&l_2+l_1(m+2)\kappa(t-t_0;T_c)+(m+1)(m+2)\kappa(t-t_0;T_c)^2, \label{Eq:Krstic}
\end{align}
where $\alpha=l_1=l_2=m=1$. The convergence of~\eqref{Eq:Krstic} under different initial conditions is shown in the first row of Figure~\ref{fig:Ex1}. Notice, that convergence is obtained exactly at $T_c$. For our algorithm consider $\alpha=1$, $L(t-t_0)=0.1\exp{-(t-t_0)}$ and $\mathcal{M}=6$. Notice that, $|\ddot{y}(t)|\leq L(t-t_0)$,  $\mathcal{L}(\tau)=0.1\alpha^2T_c^2 \exp{-T_c+T_c\exp{-\alpha\tau}-2\alpha \tau}$ and therefore $\frac{1}{\mathcal{L}(\tau)}\left|\frac{d\mathcal{L}(\tau)}{d\tau}\right|=\alpha (T_c \exp{-\alpha\tau}+2)\leq\alpha (T_c +2)\leq \mathcal{M}$. The convergence of our algorithm under different initial conditions is illustrated in the second row of Figure~\ref{fig:Ex1}. Notice that zero differentiation error is obtained before the desired time given by $T_c=1$. To illustrate how to maintain a bounded gain, consider the workaround proposed in Remark~\ref{Rem:BoundedGains}, which is the same described in
Section~II.A of~\cite{Holloway2019}. For illustrative purposes we choose a bound for $\kappa(t-t_0)$ of $10$. For this case, the simulation of the algorithm~\eqref{Eq:Krstic} is shown in the first row of Figure~\ref{fig:Ex1Bounded}, whereas the simulation of our algorithm is shown in the second row of Figure~\ref{fig:Ex1Bounded}. Notice that, for any nonzero initial condition, only convergence to a neighbourhood of the signals is obtained in the prescribed time with the algorithm under comparison~\cite{Holloway2019}, but zero error cannot be obtained. In fact, the size of such neighbourhood tends to infinite as $\|e_0\|\to+\infty$, whereas in our algorithm there is a neighbourhood around the origin of $e_0$, where predefined convergence is still obtained. By selecting the bound for $\kappa(t-t_0)$, the size of such neighbourhood can be made arbitrarily large.

\begin{figure*}
    \centering
    \def\svgwidth{18cm}
\begingroup%
  \makeatletter%
  \providecommand\color[2][]{%
    \errmessage{(Inkscape) Color is used for the text in Inkscape, but the package 'color.sty' is not loaded}%
    \renewcommand\color[2][]{}%
  }%
  \providecommand\transparent[1]{%
    \errmessage{(Inkscape) Transparency is used (non-zero) for the text in Inkscape, but the package 'transparent.sty' is not loaded}%
    \renewcommand\transparent[1]{}%
  }%
  \providecommand\rotatebox[2]{#2}%
  \newcommand*\fsize{\dimexpr\f@size pt\relax}%
  \newcommand*\lineheight[1]{\fontsize{\fsize}{#1\fsize}\selectfont}%
  \ifx\svgwidth\undefined%
    \setlength{\unitlength}{1080bp}%
    \ifx\svgscale\undefined%
      \relax%
    \else%
      \setlength{\unitlength}{\unitlength * \real{\svgscale}}%
    \fi%
  \else%
    \setlength{\unitlength}{\svgwidth}%
  \fi%
  \global\let\svgwidth\undefined%
  \global\let\svgscale\undefined%
  \makeatother%
  \begin{picture}(1,0.46666667)%
    \lineheight{1}%
    \setlength\tabcolsep{0pt}%
    \put(0,0){\includegraphics[width=\unitlength,page=1]{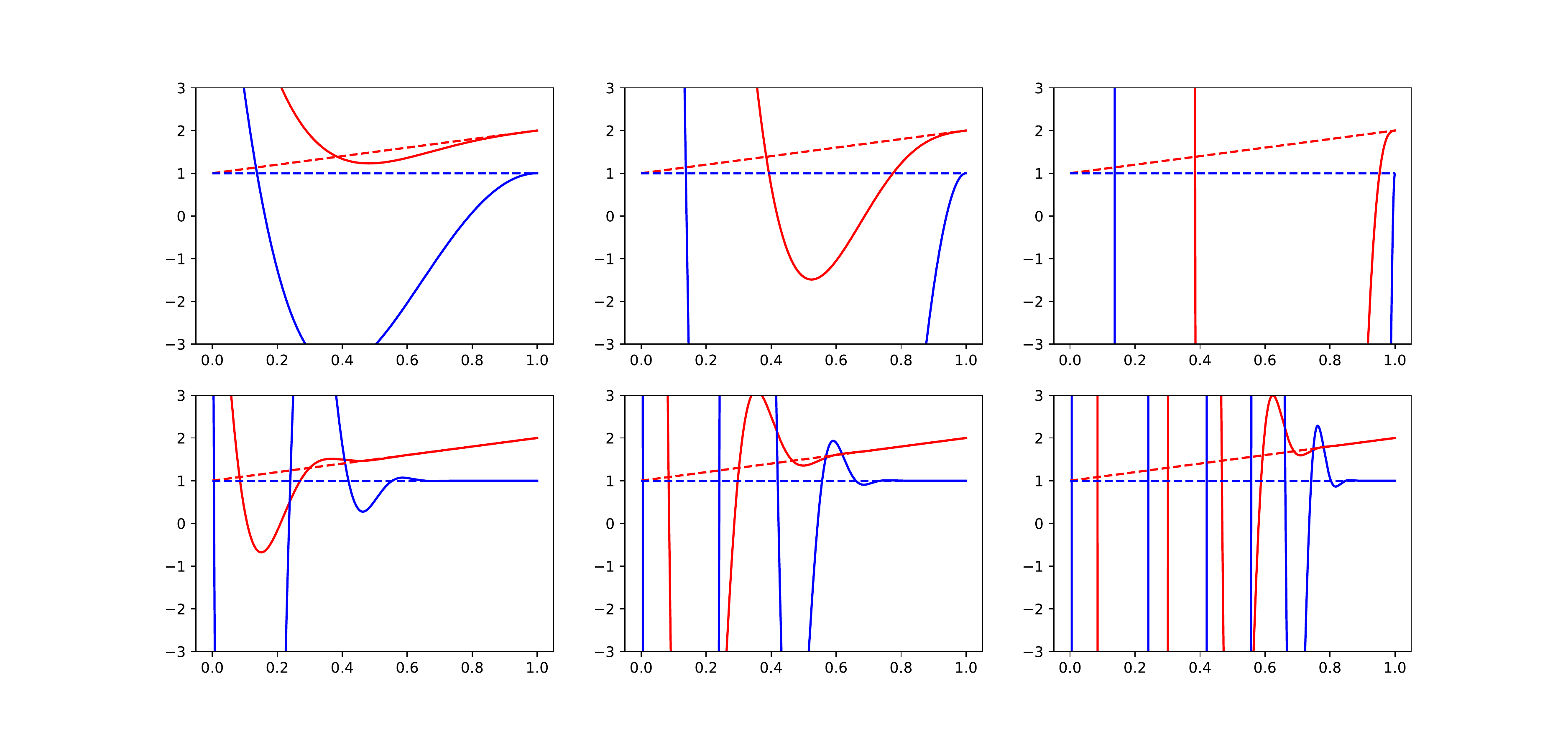}}%
    \put(0,0){\includegraphics[width=\unitlength,page=1]{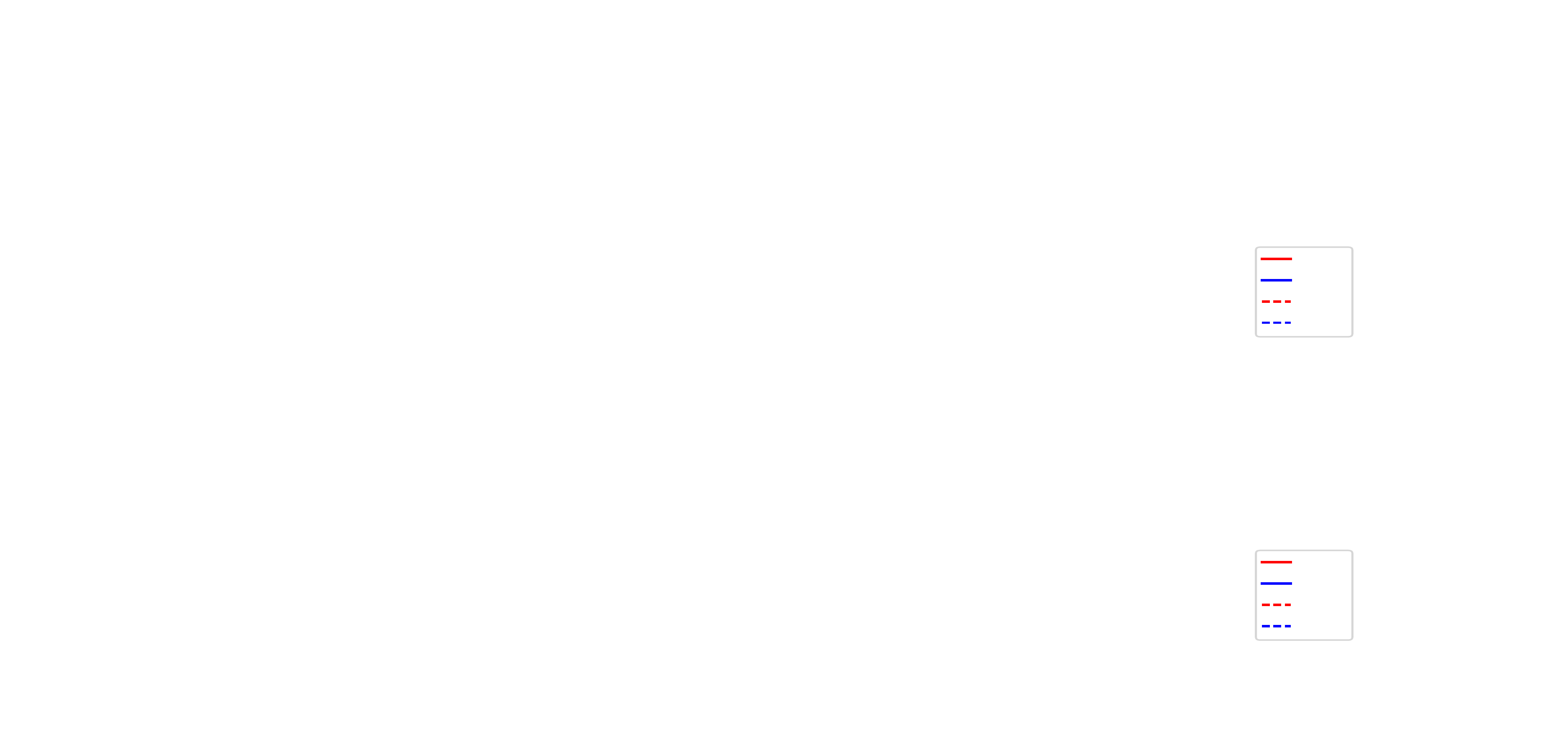}}%
    \footnotesize{
    \put(0.82463322,0.2986452){\color[rgb]{0,0,0}\makebox(0,0)[lt]{\lineheight{1.25}\smash{\begin{tabular}[t]{l}$z_0(t)$\end{tabular}}}}%
    \put(0.82463322,0.2861452){\color[rgb]{0,0,0}\makebox(0,0)[lt]{\lineheight{1.25}\smash{\begin{tabular}[t]{l}$z_1(t)$\end{tabular}}}}%
    \put(0.82463322,0.27225631){\color[rgb]{0,0,0}\makebox(0,0)[lt]{\lineheight{1.25}\smash{\begin{tabular}[t]{l}$y(t)$\end{tabular}}}}%
    \put(0.82463322,0.25836743){\color[rgb]{0,0,0}\makebox(0,0)[lt]{\lineheight{1.25}\smash{\begin{tabular}[t]{l}$\dot{y}(t)$\end{tabular}}}}%
    \put(0.82463322,0.10558963){\color[rgb]{0,0,0}\makebox(0,0)[lt]{\lineheight{1.25}\smash{\begin{tabular}[t]{l}$z_0(t)$\end{tabular}}}}%
    \put(0.82463322,0.09308963){\color[rgb]{0,0,0}\makebox(0,0)[lt]{\lineheight{1.25}\smash{\begin{tabular}[t]{l}$z_1(t)$\end{tabular}}}}%
    \put(0.82463322,0.07920074){\color[rgb]{0,0,0}\makebox(0,0)[lt]{\lineheight{1.25}\smash{\begin{tabular}[t]{l}$y(t)$\end{tabular}}}}%
    \put(0.82463322,0.06531186){\color[rgb]{0,0,0}\makebox(0,0)[lt]{\lineheight{1.25}\smash{\begin{tabular}[t]{l}$\dot{y}(t)$\end{tabular}}}}%
    \put(0.10154998,0.12145524){\color[rgb]{0,0,0}\rotatebox{90}{\makebox(0,0)[lt]{\lineheight{1.25}\smash{\begin{tabular}[t]{l}Ours\end{tabular}}}}}%
    \put(0.10154998,0.31589969){\color[rgb]{0,0,0}\rotatebox{90}{\makebox(0,0)[lt]{\lineheight{1.25}\smash{\begin{tabular}[t]{l}Holloway et al.~\cite{Holloway2019}\end{tabular}}}}}%
    \put(0.15908303,0.41515092){\color[rgb]{0,0,0}\makebox(0,0)[lt]{\lineheight{1.25}\smash{\begin{tabular}[t]{l}$z_0(t_0)=z_0(t_0)=1e1$\end{tabular}}}}%
    \put(0.44380527,0.41515092){\color[rgb]{0,0,0}\makebox(0,0)[lt]{\lineheight{1.25}\smash{\begin{tabular}[t]{l}$z_0(t_0)=z_0(t_0)=1e2$\end{tabular}}}}%
    \put(0.71741638,0.41515092){\color[rgb]{0,0,0}\makebox(0,0)[lt]{\lineheight{1.25}\smash{\begin{tabular}[t]{l}$z_0(t_0)=z_0(t_0)=1e4$\end{tabular}}}}%
    \put(0.23,0.02){\color[rgb]{0,0,0}\makebox(0,0)[lt]{\lineheight{1.25}\smash{\begin{tabular}[t]{l}time\end{tabular}}}}%
    \put(0.5,0.02){\color[rgb]{0,0,0}\makebox(0,0)[lt]{\lineheight{1.25}\smash{\begin{tabular}[t]{l}time\end{tabular}}}}%
    \put(0.78,0.02){\color[rgb]{0,0,0}\makebox(0,0)[lt]{\lineheight{1.25}\smash{\begin{tabular}[t]{l}time\end{tabular}}}}%
    }
  \end{picture}%
\endgroup%
    \caption{Simulation of Example~\ref{Ex:Krstic}, i.e., online differentiation of the signal $y(t)=t+1$.}
    \label{fig:Ex1}
\end{figure*}
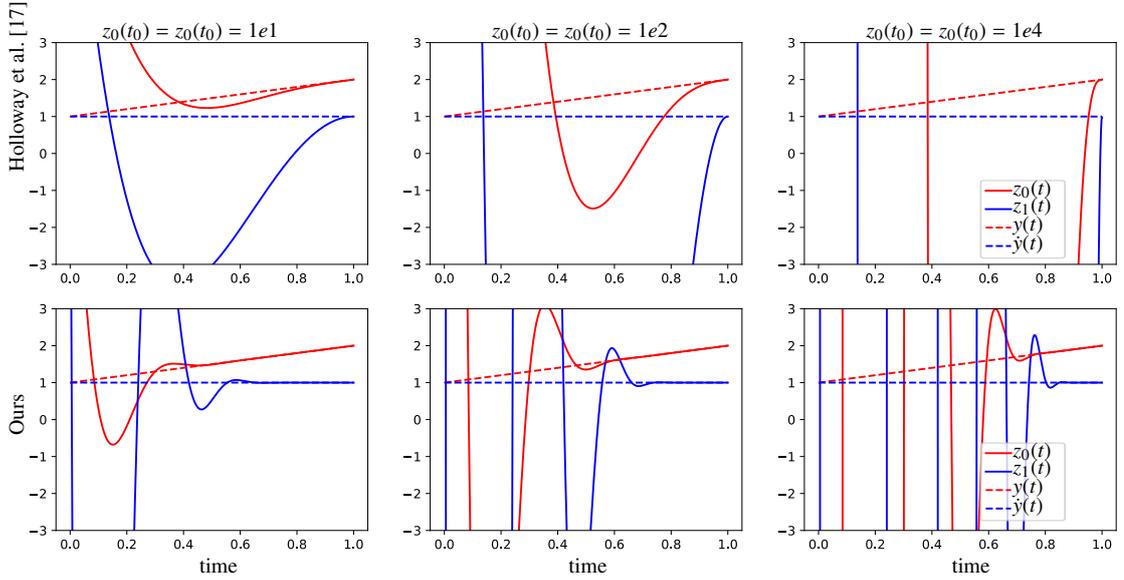

\begin{figure}
    \centering
    \def\svgwidth{18cm}
\begingroup%
  \makeatletter%
  \providecommand\color[2][]{%
    \errmessage{(Inkscape) Color is used for the text in Inkscape, but the package 'color.sty' is not loaded}%
    \renewcommand\color[2][]{}%
  }%
  \providecommand\transparent[1]{%
    \errmessage{(Inkscape) Transparency is used (non-zero) for the text in Inkscape, but the package 'transparent.sty' is not loaded}%
    \renewcommand\transparent[1]{}%
  }%
  \providecommand\rotatebox[2]{#2}%
  \newcommand*\fsize{\dimexpr\f@size pt\relax}%
  \newcommand*\lineheight[1]{\fontsize{\fsize}{#1\fsize}\selectfont}%
  \ifx\svgwidth\undefined%
    \setlength{\unitlength}{1080bp}%
    \ifx\svgscale\undefined%
      \relax%
    \else%
      \setlength{\unitlength}{\unitlength * \real{\svgscale}}%
    \fi%
  \else%
    \setlength{\unitlength}{\svgwidth}%
  \fi%
  \global\let\svgwidth\undefined%
  \global\let\svgscale\undefined%
  \makeatother%
  \begin{picture}(1,0.46666667)%
    \lineheight{1}%
    \setlength\tabcolsep{0pt}%
    \put(0,0){\includegraphics[width=\unitlength,page=1]{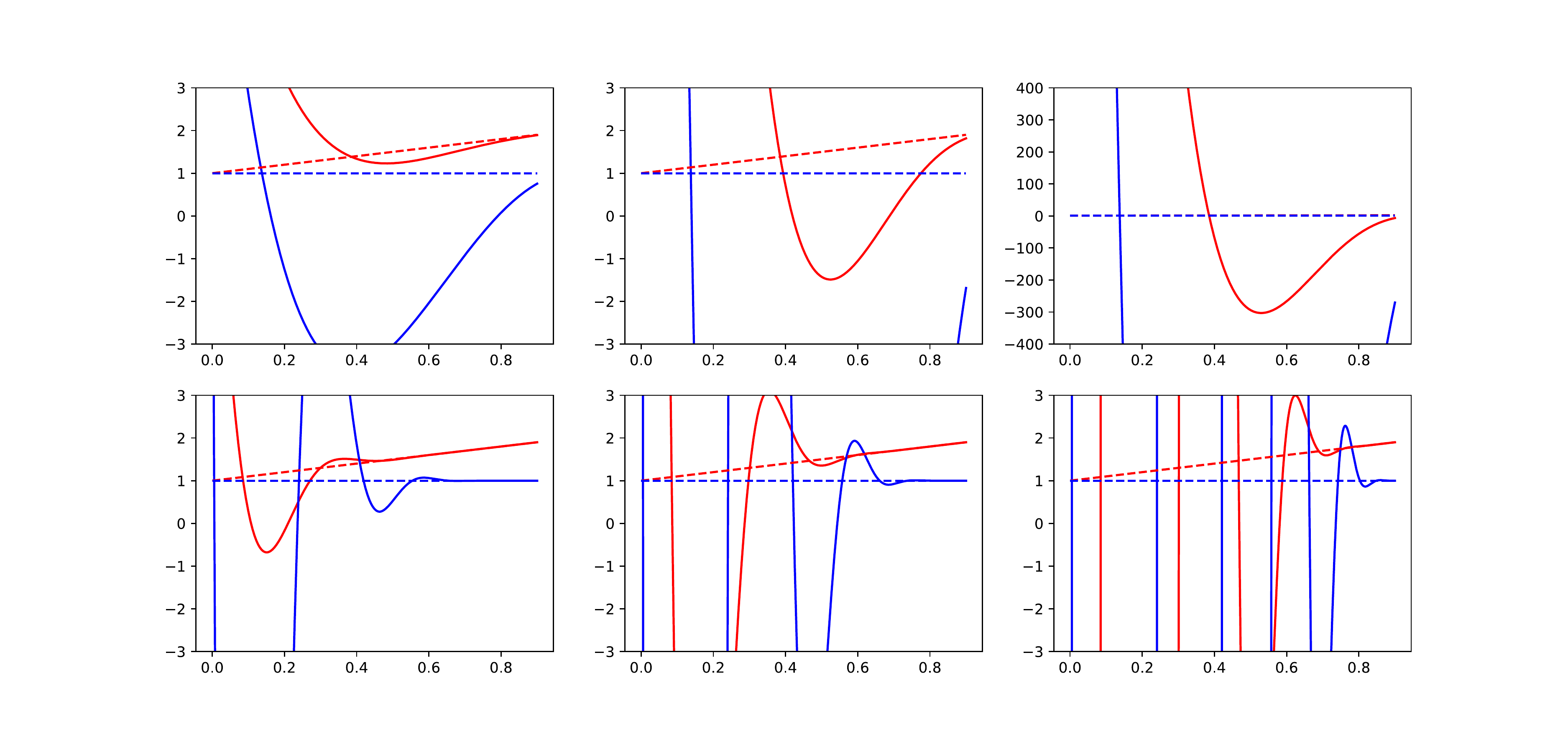}}%
    \put(0,0){\includegraphics[width=\unitlength,page=1]{figures/Legend.pdf}}%
    \footnotesize{
    \put(0.82463322,0.2986452){\color[rgb]{0,0,0}\makebox(0,0)[lt]{\lineheight{1.25}\smash{\begin{tabular}[t]{l}$z_0(t)$\end{tabular}}}}%
    \put(0.82463322,0.2861452){\color[rgb]{0,0,0}\makebox(0,0)[lt]{\lineheight{1.25}\smash{\begin{tabular}[t]{l}$z_1(t)$\end{tabular}}}}%
    \put(0.82463322,0.27225631){\color[rgb]{0,0,0}\makebox(0,0)[lt]{\lineheight{1.25}\smash{\begin{tabular}[t]{l}$y(t)$\end{tabular}}}}%
    \put(0.82463322,0.25836743){\color[rgb]{0,0,0}\makebox(0,0)[lt]{\lineheight{1.25}\smash{\begin{tabular}[t]{l}$\dot{y}(t)$\end{tabular}}}}%
    \put(0.82463322,0.10558963){\color[rgb]{0,0,0}\makebox(0,0)[lt]{\lineheight{1.25}\smash{\begin{tabular}[t]{l}$z_0(t)$\end{tabular}}}}%
    \put(0.82463322,0.09308963){\color[rgb]{0,0,0}\makebox(0,0)[lt]{\lineheight{1.25}\smash{\begin{tabular}[t]{l}$z_1(t)$\end{tabular}}}}%
    \put(0.82463322,0.07920074){\color[rgb]{0,0,0}\makebox(0,0)[lt]{\lineheight{1.25}\smash{\begin{tabular}[t]{l}$y(t)$\end{tabular}}}}%
    \put(0.82463322,0.06531186){\color[rgb]{0,0,0}\makebox(0,0)[lt]{\lineheight{1.25}\smash{\begin{tabular}[t]{l}$\dot{y}(t)$\end{tabular}}}}%
    \put(0.10154998,0.12145524){\color[rgb]{0,0,0}\rotatebox{90}{\makebox(0,0)[lt]{\lineheight{1.25}\smash{\begin{tabular}[t]{l}Ours\end{tabular}}}}}%
    \put(0.10154998,0.31589969){\color[rgb]{0,0,0}\rotatebox{90}{\makebox(0,0)[lt]{\lineheight{1.25}\smash{\begin{tabular}[t]{l}Holloway et al.~\cite{Holloway2019}\end{tabular}}}}}%
    \put(0.15908303,0.41515092){\color[rgb]{0,0,0}\makebox(0,0)[lt]{\lineheight{1.25}\smash{\begin{tabular}[t]{l}$z_0(t_0)=z_0(t_0)=1e1$\end{tabular}}}}%
    \put(0.44380527,0.41515092){\color[rgb]{0,0,0}\makebox(0,0)[lt]{\lineheight{1.25}\smash{\begin{tabular}[t]{l}$z_0(t_0)=z_0(t_0)=1e2$\end{tabular}}}}%
    \put(0.71741638,0.41515092){\color[rgb]{0,0,0}\makebox(0,0)[lt]{\lineheight{1.25}\smash{\begin{tabular}[t]{l}$z_0(t_0)=z_0(t_0)=1e4$\end{tabular}}}}%
    \put(0.23,0.02){\color[rgb]{0,0,0}\makebox(0,0)[lt]{\lineheight{1.25}\smash{\begin{tabular}[t]{l}time\end{tabular}}}}%
    \put(0.5,0.02){\color[rgb]{0,0,0}\makebox(0,0)[lt]{\lineheight{1.25}\smash{\begin{tabular}[t]{l}time\end{tabular}}}}%
    \put(0.78,0.02){\color[rgb]{0,0,0}\makebox(0,0)[lt]{\lineheight{1.25}\smash{\begin{tabular}[t]{l}time\end{tabular}}}}%
    }
  \end{picture}%
\endgroup%
    \caption{Simulation of Example~\ref{Ex:Krstic}, i.e., online differentiation of the signal $y(t)=t+1$. Bounding the \textit{TBG} gain for practical scenarios. For illustration purposes the \textit{TBG} gain is maintained below $\kappa(t-t_0)\leq10$, using the workaround suggested in Remark~\ref{Rem:BoundedGains}.}
    \label{fig:Ex1Bounded}
\end{figure}
\end{example}

Notice that the algorithm under comparison~\cite{Holloway2019} cannot be applied for the first-order exact differentiator problem of sine functions, which is our next example. Compared with autonomous predefined-time first-order differentiators~\cite{Cruz-Zavala2011,Seeber2020ExactBound}, whose predefined \textit{UBST} is conservative, here we show that the slack in our predefined \textit{UBST} is significantly reduced.

\begin{example}
\label{Ex:Sine}
Let $y(t)=\sin(t)$. Thus, $|\ddot{y}(t)|\leq 1$. Notice that $L(t-t_0)=1$. For comparison, consider the algorithm in~\cite{Seeber2020ExactBound}, i.e., algorithm~\eqref{Eq:Diff} with $h_i(w,t;t_c)=k_i\nu_i(w)$, $i=1,2$, where
\begin{align}
    \nu_1(w)&=\sgn{w}^{\frac{1}{2}}+k_3^2\sgn{w}^{\frac{3}{2}}\\
    \nu_2(w)&=\sgn{w}^{0}+4k_3^2w+k_3^4\sgn{w}^{2}, \label{Eq:Seeber}
\end{align}
where $k_1=4\sqrt{L}$, $k_2=2L$ and $k_3=\frac{9.8}{T_c\sqrt{L}}$. The convergence of~\eqref{Eq:Seeber}, under different initial conditions is shown in the first row of Figure~\ref{fig:Ex2States}. For our simulation we take $\alpha=0.3$ and $M=1$. Notice that $\mathcal{M}$ should satisfy $\mathcal{M}\geq 2\alpha$. Thus, we take $\mathcal{M}=1$. The convergence of our algorithm under different initial conditions is shown in the second row of Figure~\ref{fig:Ex2States}. The convergence of the error signals is shown in Figure~\ref{fig:Ex2Errors}. Notice in the first two rows of Figure~\ref{fig:Ex2Errors}, that the convergence of the algorithm~\eqref{Eq:Seeber} occurs before $0.5s$, but the predefined one is $T_c$. Thus it has a slack of $0.5s$. As can be seen, in the last two rows of Figure~\ref{fig:Ex2Errors}, such slack is significantly reduced in our algorithm (in fact, it can be made arbitrarily small). An advantage of reducing such slack is that the maximum differentiation error is significantly reduced (in this simulation, our maximum differentiation error results in several orders of magnitude lower), an important feature when the differentiator is in closed loop with a controller.

\begin{figure}
    \centering
    \def\svgwidth{18cm}
\begingroup%
  \makeatletter%
  \providecommand\color[2][]{%
    \errmessage{(Inkscape) Color is used for the text in Inkscape, but the package 'color.sty' is not loaded}%
    \renewcommand\color[2][]{}%
  }%
  \providecommand\transparent[1]{%
    \errmessage{(Inkscape) Transparency is used (non-zero) for the text in Inkscape, but the package 'transparent.sty' is not loaded}%
    \renewcommand\transparent[1]{}%
  }%
  \providecommand\rotatebox[2]{#2}%
  \newcommand*\fsize{\dimexpr\f@size pt\relax}%
  \newcommand*\lineheight[1]{\fontsize{\fsize}{#1\fsize}\selectfont}%
  \ifx\svgwidth\undefined%
    \setlength{\unitlength}{1080bp}%
    \ifx\svgscale\undefined%
      \relax%
    \else%
      \setlength{\unitlength}{\unitlength * \real{\svgscale}}%
    \fi%
  \else%
    \setlength{\unitlength}{\svgwidth}%
  \fi%
  \global\let\svgwidth\undefined%
  \global\let\svgscale\undefined%
  \makeatother%
  \begin{picture}(1,0.46666667)%
    \lineheight{1}%
    \setlength\tabcolsep{0pt}%
    \put(0,0){\includegraphics[width=\unitlength,page=1]{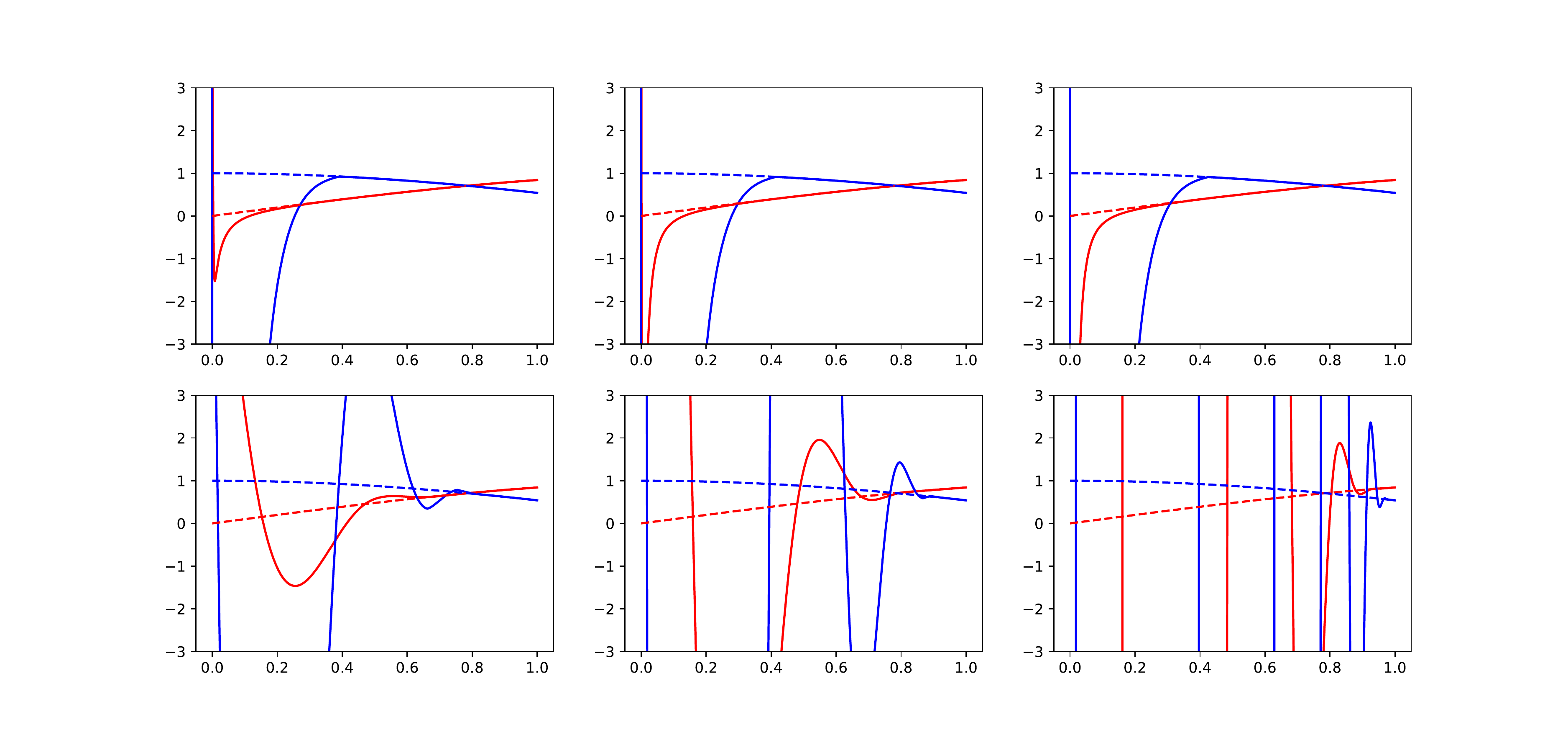}}%
    \put(0,0){\includegraphics[width=\unitlength,page=1]{figures/Legend.pdf}}%
    \footnotesize{
    \put(0.82463322,0.2986452){\color[rgb]{0,0,0}\makebox(0,0)[lt]{\lineheight{1.25}\smash{\begin{tabular}[t]{l}$z_0(t)$\end{tabular}}}}%
    \put(0.82463322,0.2861452){\color[rgb]{0,0,0}\makebox(0,0)[lt]{\lineheight{1.25}\smash{\begin{tabular}[t]{l}$z_1(t)$\end{tabular}}}}%
    \put(0.82463322,0.27225631){\color[rgb]{0,0,0}\makebox(0,0)[lt]{\lineheight{1.25}\smash{\begin{tabular}[t]{l}$y(t)$\end{tabular}}}}%
    \put(0.82463322,0.25836743){\color[rgb]{0,0,0}\makebox(0,0)[lt]{\lineheight{1.25}\smash{\begin{tabular}[t]{l}$\dot{y}(t)$\end{tabular}}}}%
    \put(0.82463322,0.10558963){\color[rgb]{0,0,0}\makebox(0,0)[lt]{\lineheight{1.25}\smash{\begin{tabular}[t]{l}$z_0(t)$\end{tabular}}}}%
    \put(0.82463322,0.09308963){\color[rgb]{0,0,0}\makebox(0,0)[lt]{\lineheight{1.25}\smash{\begin{tabular}[t]{l}$z_1(t)$\end{tabular}}}}%
    \put(0.82463322,0.07920074){\color[rgb]{0,0,0}\makebox(0,0)[lt]{\lineheight{1.25}\smash{\begin{tabular}[t]{l}$y(t)$\end{tabular}}}}%
    \put(0.82463322,0.06531186){\color[rgb]{0,0,0}\makebox(0,0)[lt]{\lineheight{1.25}\smash{\begin{tabular}[t]{l}$\dot{y}(t)$\end{tabular}}}}%
    \put(0.10154998,0.12145524){\color[rgb]{0,0,0}\rotatebox{90}{\makebox(0,0)[lt]{\lineheight{1.25}\smash{\begin{tabular}[t]{l}Ours\end{tabular}}}}}%
    \put(0.10154998,0.31589969){\color[rgb]{0,0,0}\rotatebox{90}{\makebox(0,0)[lt]{\lineheight{1.25}\smash{\begin{tabular}[t]{l}Seeber et al.~\cite{Seeber2020ExactBound}\end{tabular}}}}}%
    \put(0.15908303,0.41515092){\color[rgb]{0,0,0}\makebox(0,0)[lt]{\lineheight{1.25}\smash{\begin{tabular}[t]{l}$z_0(t_0)=z_0(t_0)=1e1$\end{tabular}}}}%
    \put(0.44380527,0.41515092){\color[rgb]{0,0,0}\makebox(0,0)[lt]{\lineheight{1.25}\smash{\begin{tabular}[t]{l}$z_0(t_0)=z_0(t_0)=1e2$\end{tabular}}}}%
    \put(0.71741638,0.41515092){\color[rgb]{0,0,0}\makebox(0,0)[lt]{\lineheight{1.25}\smash{\begin{tabular}[t]{l}$z_0(t_0)=z_0(t_0)=1e4$\end{tabular}}}}%
    \put(0.23,0.02){\color[rgb]{0,0,0}\makebox(0,0)[lt]{\lineheight{1.25}\smash{\begin{tabular}[t]{l}time\end{tabular}}}}%
    \put(0.5,0.02){\color[rgb]{0,0,0}\makebox(0,0)[lt]{\lineheight{1.25}\smash{\begin{tabular}[t]{l}time\end{tabular}}}}%
    \put(0.78,0.02){\color[rgb]{0,0,0}\makebox(0,0)[lt]{\lineheight{1.25}\smash{\begin{tabular}[t]{l}time\end{tabular}}}}%
    }
  \end{picture}%
\endgroup%
    \caption{Simulation of Example~\ref{Ex:Krstic}, i.e., online differentiation of the signal $y(t)=\sin(t)$.}
    \label{fig:Ex2States}
\end{figure}

\begin{figure*}
    \centering
    \def\svgwidth{18cm}
\begingroup%
  \makeatletter%
  \providecommand\color[2][]{%
    \errmessage{(Inkscape) Color is used for the text in Inkscape, but the package 'color.sty' is not loaded}%
    \renewcommand\color[2][]{}%
  }%
  \providecommand\transparent[1]{%
    \errmessage{(Inkscape) Transparency is used (non-zero) for the text in Inkscape, but the package 'transparent.sty' is not loaded}%
    \renewcommand\transparent[1]{}%
  }%
  \providecommand\rotatebox[2]{#2}%
  \newcommand*\fsize{\dimexpr\f@size pt\relax}%
  \newcommand*\lineheight[1]{\fontsize{\fsize}{#1\fsize}\selectfont}%
  \ifx\svgwidth\undefined%
    \setlength{\unitlength}{1080bp}%
    \ifx\svgscale\undefined%
      \relax%
    \else%
      \setlength{\unitlength}{\unitlength * \real{\svgscale}}%
    \fi%
  \else%
    \setlength{\unitlength}{\svgwidth}%
  \fi%
  \global\let\svgwidth\undefined%
  \global\let\svgscale\undefined%
  \makeatother%
  \begin{picture}(1,0.93333333)%
    \lineheight{1}%
    \setlength\tabcolsep{0pt}%
    \put(0,0){\includegraphics[width=\unitlength,page=1]{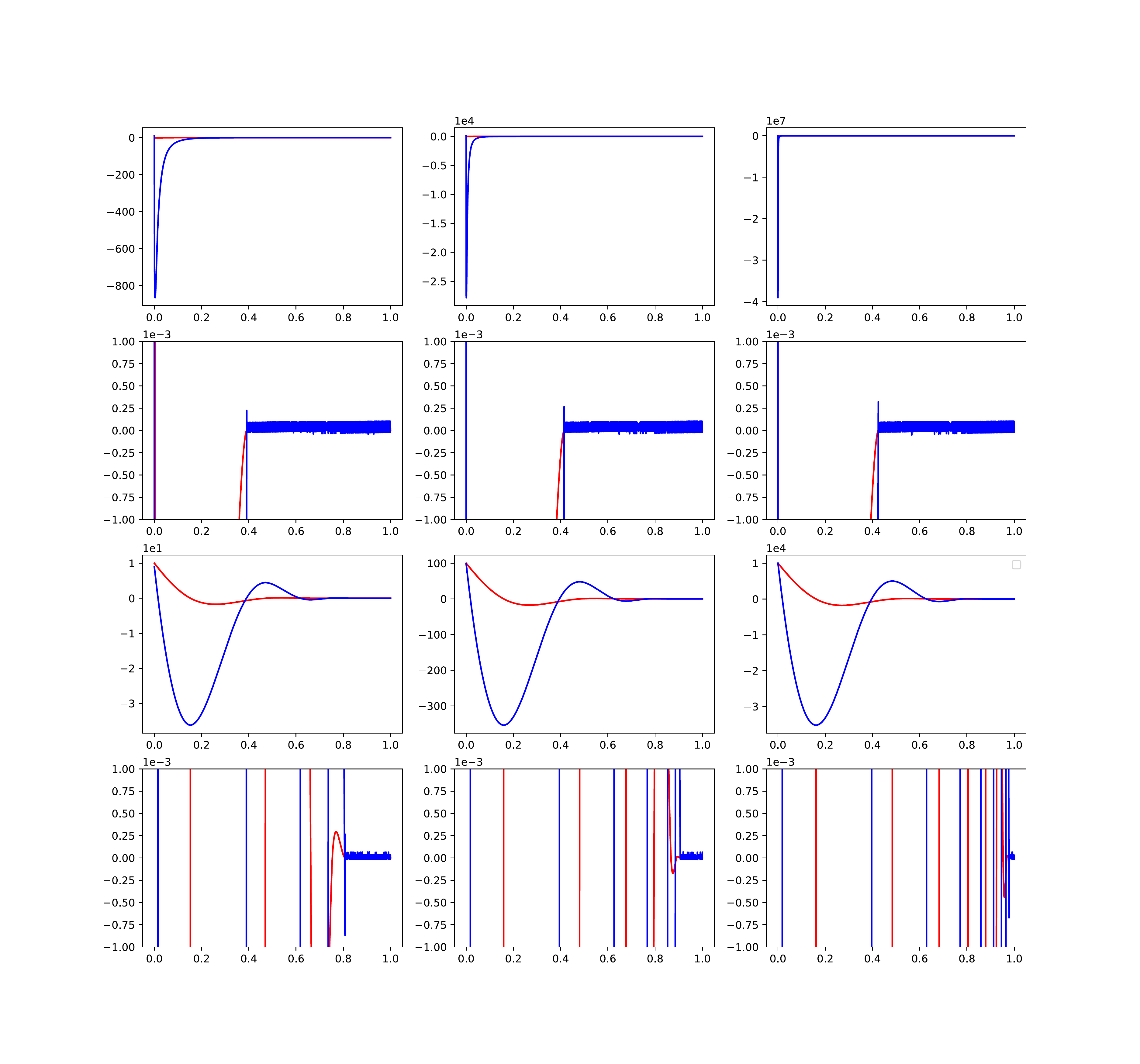}}%
    \put(0,0){\includegraphics[width=\unitlength,page=1]{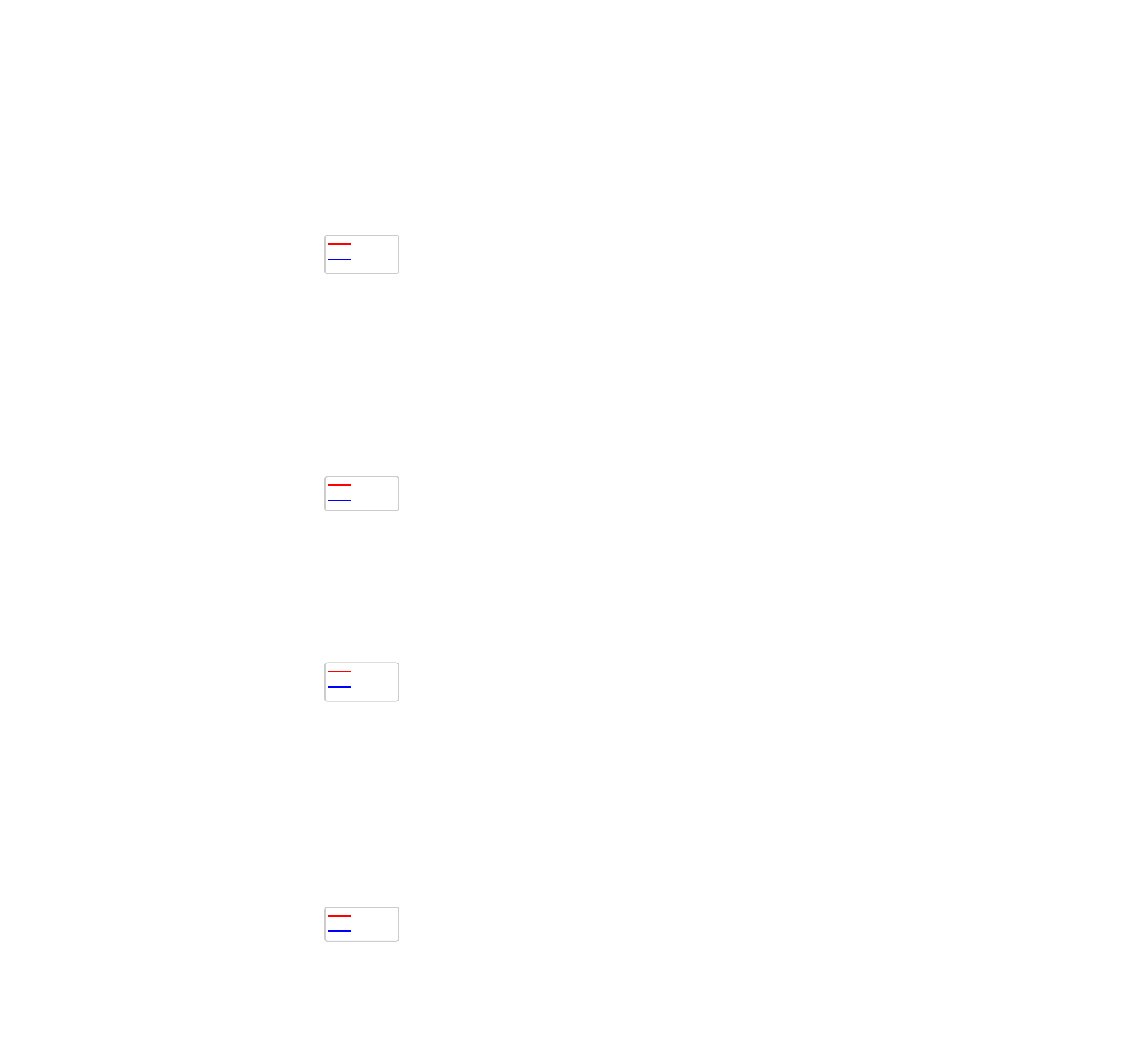}}%
    \footnotesize{
    \put(0.0831274,0.54481778){\color[rgb]{0,0,0}\rotatebox{90}{\makebox(0,0)[lt]{\lineheight{1.25}\smash{\begin{tabular}[t]{l}Seeber et al.~\cite{Seeber2020ExactBound}\end{tabular}}}}}%
    \put(0.0831274,0.73092889){\color[rgb]{0,0,0}\rotatebox{90}{\makebox(0,0)[lt]{\lineheight{1.25}\smash{\begin{tabular}[t]{l}Seeber et al.~\cite{Seeber2020ExactBound}\end{tabular}}}}}%
    \put(0.16948773,0.82694675){\color[rgb]{0,0,0}\makebox(0,0)[lt]{\lineheight{1.25}\smash{\begin{tabular}[t]{l}$z(t_0)=z(t_0)=1e1$\end{tabular}}}}%
    \put(0.44726551,0.82694675){\color[rgb]{0,0,0}\makebox(0,0)[lt]{\lineheight{1.25}\smash{\begin{tabular}[t]{l}$z(t_0)=z(t_0)=1e2$\end{tabular}}}}%
    \put(0.72504329,0.82694675){\color[rgb]{0,0,0}\makebox(0,0)[lt]{\lineheight{1.25}\smash{\begin{tabular}[t]{l}$z(t_0)=z(t_0)=1e4$\end{tabular}}}}%
    \put(0.0831274,0.16704){\color[rgb]{0,0,0}\rotatebox{90}{\makebox(0,0)[lt]{\lineheight{1.25}\smash{\begin{tabular}[t]{l}Ours\end{tabular}}}}}%
    \put(0.0831274,0.35731778){\color[rgb]{0,0,0}\rotatebox{90}{\makebox(0,0)[lt]{\lineheight{1.25}\smash{\begin{tabular}[t]{l}Ours\end{tabular}}}}}%
    \put(0.31125873,0.71611215){\color[rgb]{0,0,0}\makebox(0,0)[lt]{\lineheight{1.25}\smash{\begin{tabular}[t]{l}$e_1(t)$\end{tabular}}}}%
    \put(0.31125873,0.70361215){\color[rgb]{0,0,0}\makebox(0,0)[lt]{\lineheight{1.25}\smash{\begin{tabular}[t]{l}$e_2(t)$\end{tabular}}}}%
    \put(0.31125873,0.34111215){\color[rgb]{0,0,0}\makebox(0,0)[lt]{\lineheight{1.25}\smash{\begin{tabular}[t]{l}$e_1(t)$\end{tabular}}}}%
    \put(0.31125873,0.32861215){\color[rgb]{0,0,0}\makebox(0,0)[lt]{\lineheight{1.25}\smash{\begin{tabular}[t]{l}$e_2(t)$\end{tabular}}}}%
    \put(0.31264761,0.50395355){\color[rgb]{0,0,0}\makebox(0,0)[lt]{\lineheight{1.25}\smash{\begin{tabular}[t]{l}$e_1(t)$\end{tabular}}}}%
    \put(0.31264761,0.49284244){\color[rgb]{0,0,0}\makebox(0,0)[lt]{\lineheight{1.25}\smash{\begin{tabular}[t]{l}$e_2(t)$\end{tabular}}}}%
    \put(0.31264761,0.12756466){\color[rgb]{0,0,0}\makebox(0,0)[lt]{\lineheight{1.25}\smash{\begin{tabular}[t]{l}$e_1(t)$\end{tabular}}}}%
    \put(0.31264761,0.11645355){\color[rgb]{0,0,0}\makebox(0,0)[lt]{\lineheight{1.25}\smash{\begin{tabular}[t]{l}$e_2(t)$\end{tabular}}}}%
    \put(0.23,0.07){\color[rgb]{0,0,0}\makebox(0,0)[lt]{\lineheight{1.25}\smash{\begin{tabular}[t]{l}time\end{tabular}}}}%
    \put(0.5,0.07){\color[rgb]{0,0,0}\makebox(0,0)[lt]{\lineheight{1.25}\smash{\begin{tabular}[t]{l}time\end{tabular}}}}%
    \put(0.78,0.07){\color[rgb]{0,0,0}\makebox(0,0)[lt]{\lineheight{1.25}\smash{\begin{tabular}[t]{l}time\end{tabular}}}}%
    }
  \end{picture}%
\endgroup%
    \caption{Simulation of Example~\ref{Ex:Sine}, i.e., online differentiation of the signal $y=\sin(t)$.}
    \label{fig:Ex2Errors}
\end{figure*}
\end{example}

Notice that the algorithms in~\cite{Cruz-Zavala2011,Seeber2020ExactBound} cannot be applied for the first-order exact differentiator problem of functions where the second derivative is not bounded by a constant, such as $y(t)=2\sin\left(\frac{1}{2} t^2 \right)$. Compared with the predefined-time first-order differentiators based on time-varying gains~\cite{Levant2018GloballyGains}, our approach provides predefined-time convergence, which is shown in the following example.

\begin{example}
\label{Ex:SineSquare}
Let $y(t)=2\sin\left(\frac{1}{2} t^2\right)$. It is easy to verify that such function satisfies $$|\ddot{y}(t)|=\left|2\cos\left(\frac{1}{2}t^2\right)-2t^2\sin\left(\frac{1}{2}t^2\right)\right|\leq L(t-t_0)$$ with $L(t-t_0)=2(t-t_0)^2+\beta$, $\beta\geq 2$ and that $\frac{1}{L(t-t_0)}\left|\frac{d {L}(t-t_0)}{dt}\right|\leq M=\sqrt{\frac{2}{\beta}}$. Notice that, $$\mathcal{L}(\tau)=(2(T_c(1-\exp{-\alpha\tau}))^2+\beta)(\alpha T_c \exp{-\alpha \tau})^2$$
and
\begin{align}
\frac{1}{\mathcal{L}(\tau)}\left|\frac{d \mathcal{L}(\tau)}{d\tau}\right|=&4\alpha T_c^2\frac{\exp{-\alpha\tau}(1-\exp{-\alpha\tau})}{\beta+2T_c^2(1-\exp{-\alpha\tau})^2}+2\alpha.
\end{align}

It can be verified that the maximum of $\frac{1}{\mathcal{L}(\tau)}\left|\frac{d \mathcal{L}(\tau)}{d\tau}\right|$ is $\frac{2\alpha}{(-\beta + \sqrt{\beta (2 + \beta)}}$, which occurs at $\tau=\log\left(1 + \sqrt{\frac{\beta}{2 + \beta}}\right)$.
Thus, $\mathcal{M}$ should satisfy $\mathcal{M}\geq\frac{2\alpha}{-\beta + \sqrt{\beta (2 + \beta)}}$. 

Consider our algorithm with $\alpha=0.3$, $\beta=2$ and $\mathcal{M}=\frac{0.6}{\sqrt{8}-2}$. The convergence of our algorithm under different initial conditions is shown in Figure~\ref{fig:Ex3}, where it can be verified that the convergence is upper bounded by the desired time given by $T_c$.


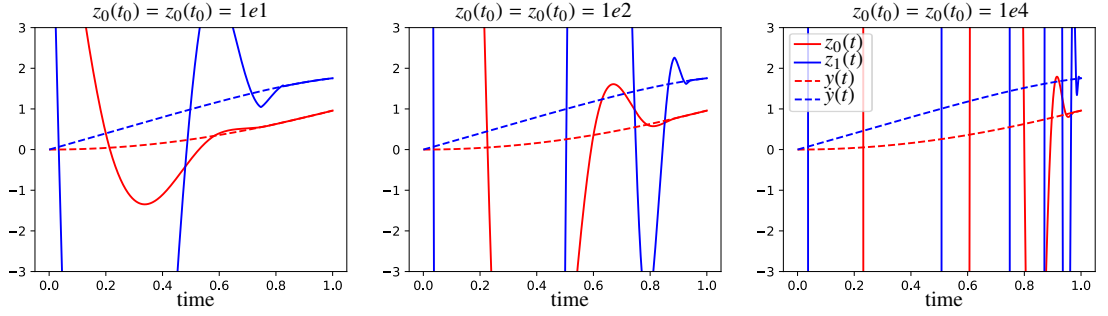
\begin{figure*}
    \centering
    \def\svgwidth{18cm}
\begingroup%
  \makeatletter%
  \providecommand\color[2][]{%
    \errmessage{(Inkscape) Color is used for the text in Inkscape, but the package 'color.sty' is not loaded}%
    \renewcommand\color[2][]{}%
  }%
  \providecommand\transparent[1]{%
    \errmessage{(Inkscape) Transparency is used (non-zero) for the text in Inkscape, but the package 'transparent.sty' is not loaded}%
    \renewcommand\transparent[1]{}%
  }%
  \providecommand\rotatebox[2]{#2}%
  \newcommand*\fsize{\dimexpr\f@size pt\relax}%
  \newcommand*\lineheight[1]{\fontsize{\fsize}{#1\fsize}\selectfont}%
  \ifx\svgwidth\undefined%
    \setlength{\unitlength}{1080bp}%
    \ifx\svgscale\undefined%
      \relax%
    \else%
      \setlength{\unitlength}{\unitlength * \real{\svgscale}}%
    \fi%
  \else%
    \setlength{\unitlength}{\svgwidth}%
  \fi%
  \global\let\svgwidth\undefined%
  \global\let\svgscale\undefined%
  \makeatother%
  \begin{picture}(1,0.23333333)%
    \lineheight{1}%
    \setlength\tabcolsep{0pt}%
    \put(0,0){\includegraphics[width=\unitlength,page=1]{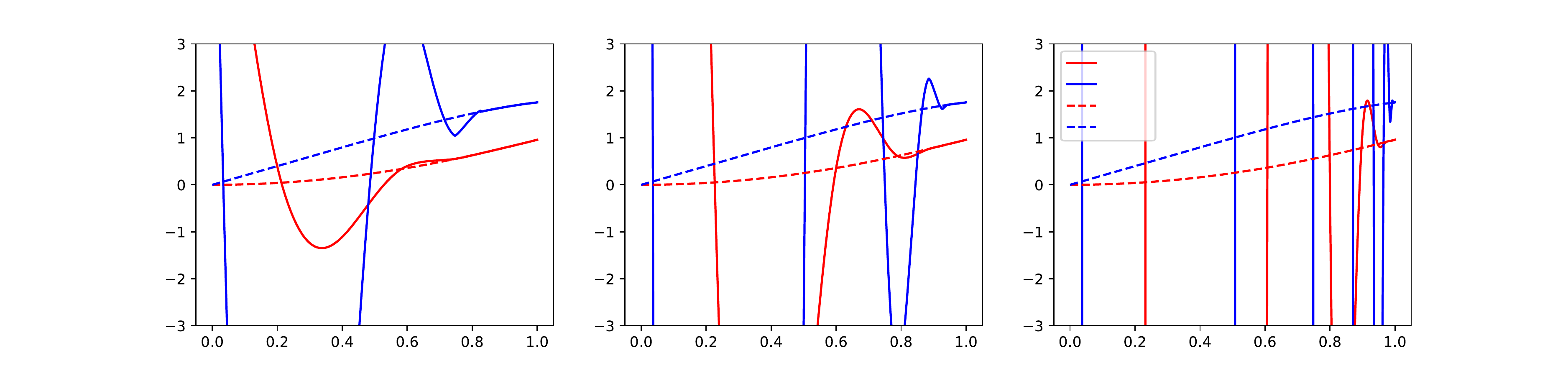}}%
    \footnotesize{
    \put(0.22863079,0.00000075){\color[rgb]{0,0,0}\makebox(0,0)[lt]{\lineheight{1.25}\smash{\begin{tabular}[t]{l}time\end{tabular}}}}%
    \put(0.50216019,0.00000075){\color[rgb]{0,0,0}\makebox(0,0)[lt]{\lineheight{1.25}\smash{\begin{tabular}[t]{l}time\end{tabular}}}}%
    \put(0.77568958,0.00000075){\color[rgb]{0,0,0}\makebox(0,0)[lt]{\lineheight{1.25}\smash{\begin{tabular}[t]{l}time\end{tabular}}}}%
    \put(0.70250459,0.19002878){\color[rgb]{0,0,0}\makebox(0,0)[lt]{\lineheight{1.25}\smash{\begin{tabular}[t]{l}$z_0(t)$\end{tabular}}}}%
    \put(0.70250459,0.17752878){\color[rgb]{0,0,0}\makebox(0,0)[lt]{\lineheight{1.25}\smash{\begin{tabular}[t]{l}$z_1(t)$\end{tabular}}}}%
    \put(0.70250459,0.16363989){\color[rgb]{0,0,0}\makebox(0,0)[lt]{\lineheight{1.25}\smash{\begin{tabular}[t]{l}$y(t)$\end{tabular}}}}%
    \put(0.70250459,0.15113989){\color[rgb]{0,0,0}\makebox(0,0)[lt]{\lineheight{1.25}\smash{\begin{tabular}[t]{l}$\dot{y}(t)$\end{tabular}}}}%
    \put(0.16724252,0.21281503){\color[rgb]{0,0,0}\makebox(0,0)[lt]{\lineheight{1.25}\smash{\begin{tabular}[t]{l}$z_0(t_0)=z_0(t_0)=1e1$\end{tabular}}}}%
    \put(0.43252029,0.21281503){\color[rgb]{0,0,0}\makebox(0,0)[lt]{\lineheight{1.25}\smash{\begin{tabular}[t]{l}$z_0(t_0)=z_0(t_0)=1e2$\end{tabular}}}}%
    \put(0.7227981,0.21281503){\color[rgb]{0,0,0}\makebox(0,0)[lt]{\lineheight{1.25}\smash{\begin{tabular}[t]{l}$z_0(t_0)=z_0(t_0)=1e4$\end{tabular}}}}%
    }
  \end{picture}%
\endgroup%
    \caption{Simulation of Example~\ref{Ex:SineSquare}, i.e., online differentiation of the signal $y=2\sin\left(\frac{1}{2}t^2\right)$}
    \label{fig:Ex3}
\end{figure*}
\end{example}

\section{Conclusion}
\label{Sec:Conclu}

In this paper, we introduced a predefined-time first-order exact differentiator algorithm for the case where the second derivative of a signal is bounded by a known time-varying function. Our approach redesigns the algorithm proposed by Levant and Livne, which is based on time-varying gains, by incorporating a time-varying gain known as \textit{TBG} gain. To our best knowledge, our approach is the first predefined-time first-order differentiator for such class of functions. We presented numerical examples highlighting the contribution with respect to state-of-the-art algorithms.

Our future work is motivated by~\cite{Fraguela2012DesignNoise}, to consider the optimal selection of the parameters in the presence of noise.

\section*{Acknowledge}
The authors would like to thank Dr. Marco Tulio Angulo from UNAM, Juriquilla, Mexico, for his feedback on the organization of the manuscript.

\bibliographystyle{abbrv}
\bibliography{biblio}
\end{document}